\documentclass[12pt,american,reqno]{amsart}

\usepackage{geometry}
\usepackage{amstext}
\usepackage{amsthm}
\usepackage{amssymb}
\usepackage{stmaryrd}
\usepackage{graphicx}
\usepackage{multicol}
\usepackage{framed}

\usepackage{caption}
\makeatletter

\usepackage{hyperref}
\hypersetup{
colorlinks=true,
linkcolor=blue,
urlcolor=blue,
citecolor=blue
}

\numberwithin{equation}{section}
\theoremstyle{plain}
\newtheorem{thm}{\protect\theoremname}[section]
\theoremstyle{definition}
\newtheorem{defn}[thm]{\protect\definitionname}
\theoremstyle{definition}
\newtheorem{example}[thm]{\protect\examplename}
\theoremstyle{plain}
\newtheorem{prop}[thm]{\protect\propositionname}
\theoremstyle{remark}
\newtheorem{rem}[thm]{\protect\remarkname}

\usepackage{graphicx,cases,mathtools}
\usepackage{tikz}

\newtheoremstyle{plain}
  {5pt}   % ABOVESPACE
  {5pt plus 1pt minus 1pt}   % BELOWSPACE
  {}  % BODYFONT
  {0pt}       % INDENT (empty value is the same as 0pt)
  {\bfseries} % HEADFONT
  {.}         % HEADPUNCT
  {5pt plus 1pt minus 1pt} % HEADSPACE
  {}          % CUSTOM-HEAD-SPEC

\usepackage{enumitem}
\setlist{itemsep=0pt,topsep=0pt,parsep=1pt,partopsep=0pt}

\makeatother

\usepackage{babel}
\usepackage{listings}
\providecommand{\definitionname}{Definition}
\providecommand{\examplename}{Example}
\providecommand{\propositionname}{Proposition}
\providecommand{\remarkname}{Remark}
\providecommand{\theoremname}{Theorem}

\begin{document}
\global\long\def\Aut{\operatorname{Aut}}%
\global\long\def\pStab{\operatorname{pStab}}%
\global\long\def\Fix{\operatorname{Fix}}%
\global\long\def\span{\operatorname{span}}%

\title[Invariant Polydiagonal Subspaces and Constraint Programming]{Invariant Polydiagonal Subspaces of Matrices and Constraint Programming}
\author{John M. Neuberger, N\'andor Sieben, James W. Swift}

\ifdefined\FirstPage
\thanks{$\dagger$ corresponding author, ORCID 0000-0002-9433-9456}
\else
\curraddr{Northern Arizona University, Department of Mathematics and Statistics, Flagstaff, AZ 86011-5717, USA}

\email{John.Neuberger@nau.edu, Nandor.Sieben@nau.edu, Jim.Wwift@nau.edu}
\fi

\thanks{\today}



\begin{abstract}
In a polydiagonal subspace of the Euclidean space, certain components of the vectors are equal (synchrony) or opposite (anti-synchrony). Polydiagonal subspaces invariant under a matrix have many applications in graph theory and dynamical systems, especially coupled cell networks. We describe invariant polydiagonal subspaces in terms of coloring vectors. This approach gives an easy formulation of a constraint satisfaction problem for finding invariant polydiagonal subspaces. Solving the resulting problem with existing state-of-the-art constraint solvers greatly outperforms the currently known algorithms.
\end{abstract}

\maketitle

\ifdefined\FirstPage
Department of Mathematics and Statistics, Northern Arizona University, 

Flagstaff, AZ 86011-5717, USA

\medskip

Email: John.Neuberger@nau.edu, Nandor.Sieben@nau.edu, Jim.Swift@nau.edu$^\dagger$

\else

\fi

\section{Introduction}
Invariant subspaces of matrices play an important role in many applications. An invariant synchrony subspace is a special kind of invariant subspace whereby certain components of the vectors in the subspace are equal. Invariant synchrony subspaces and their corresponding partitions have a surprisingly large number of applications in graph theory, network theory, mathematical biology, and other areas of mathematics. Applications include equitable and almost equitable partitions of graph vertices, balanced and exo-balanced partitions of coupled cell networks, coset partitions of Cayley graphs, network controllability, and differential
equations, as listed in \cite{NSS7}.

An invariant anti-synchrony subspace is an invariant subspace whereby certain components of the vectors in the subspace are opposite.
Applications can be found in many settings, including multi-agent systems, chaotic oscillators, optically coupled semiconductor lasers, and neural networks, as found in \cite{NiSS}.

Finding all invariant synchrony subspaces is an NP-complete problem.
Algorithms that work for small matrices can be found in \cite{KameiLattice,NSS7,NiSS_companion}. Another approach based on eigenvectors \cite{Aguiar&Dias,AD2021} can potentially analyze much larger matrices, but the algorithm is complex and would be hard to implement on a computer. Furthermore, computing the eigenvalues of large matrices requires inexact floating-point computations. 

Polydiagonal subspaces are either synchrony or anti-synchrony subspaces.
Since there are many more of these subspaces than synchrony subspaces, 
computing all invariant polydiagonal subspaces is much more computationally intensive. This paper presents a constraint programming approach to find all invariant polydiagonal subspaces.
Our code greatly outperforms implementations of the previously listed algorithms for computing synchrony subspaces. This makes it possible to handle much larger matrices. 

This breakthrough is achieved with the help of state-of-the-art constraint solvers \cite{z3, cp_optimizer_2023, cpsatlp}. To use these solvers, the conditions for being both polydiagonal and invariant need to be translated to constraints. These conditions are enforced using coloring vectors, which encode the components that are equal or opposite.

Our main motivation for finding invariant polydiagonal subspaces is taken
from our efforts \cite{NSS3, NSS5, nss8} to find and understand bifurcations of solutions to differential equations. The Bifurcation Lemma for Invariant Subspaces (BLIS) \cite{nss8} assumes the existence of nested invariant subspaces to prove the existence of bifurcating solutions in many cases. The code available at \cite{NSS8_companion} contains an implementation of the corresponding BLIS algorithm which requires the invariant polydiagonal subspaces as input, 
both the 
anomalous invariant subspaces \cite{NSS3}
and 
the ubiquitous, well-understood fixed point subspaces. 
The current paper shows how to compute these invariant subspaces.
\section{Preliminaries}

We recall some terminology from \cite{NiSS}. The definition in \cite{AD2021} is also relevant. This provides the precise mathematical definition of synchrony and anti-synchrony subspaces.
A \emph{partial involution} on a set $X$ is a partial function $x\mapsto x^{*}:\tilde{X}\to\tilde{X}$ that is its own inverse. 
Note that $\tilde{X}\subseteq X$ and $x^{**}=x$. 
A \emph{tagged partition} of a finite set $C$ is a partition $\mathcal{P}$ of $C$ together with a partial involution $P\mapsto  P^{*}:\tilde{\mathcal{P}}\to\tilde{\mathcal{P}}$ on $\mathcal{P}$ that has at most one fixed point. 
The equivalence class of $i\in C$ is denoted by $[i]$, so that $i\in[i]\in\mathcal{P}$. 
The \emph{polydiagonal subspace of a tagged partition} $\mathcal{P}$ of $C=\{1,\ldots,n\}$ is the subspace 
\[
\Delta_{\mathcal{P}}:=\{x\in\mathbb{R}^{n}\mid x_{i}=x_{j}\text{ if }[i]=[j]\text{ and }x_{i}=-x_{j}\text{ if }[i]^{*}=[j]\}
\]
of $\mathbb{R}^{n}$. Note that if $[i]^{*}=[i]$, then $x_{i}=0$
for all $x\in\Delta_{\mathcal{P}}$. A subspace $\Delta$ of $\mathbb{R}^{n}$ is called \emph{polydiagonal} if $\Delta=\Delta_{\mathcal{P}}$ for some tagged partition $\mathcal{P}$. The polydiagonal subspace $\Delta_{\mathcal{P}}$ is called a \emph{synchrony subspace} if the partial involution is the empty function. Otherwise, it is called an \emph{anti-synchrony subspace}.
In the next section, we describe polydiagonal subspaces in terms of coloring vectors.

For a graph with vertex set $\{1,\ldots, n\}$, the \emph{adjacency matrix} is the $n\times n$ matrix $A$ for which $A_{i,j}=1$ if there is an edge between $i$ and $j$ and $A_{i,j}=0$ otherwise. The \emph{Laplacian matrix} of the graph is $L = D-A$, where $D$ is the diagonal matrix consisting of the degrees of the vertices.

\section{Coloring vectors}

Coloring vectors were introduced in \cite{NSS7} to encode partitions.
We generalize this notion for tagged partitions.

\begin{defn}
\label{def:colVec}An element $c$ of $\mathbb{Z}^{n}$ is called
a \emph{coloring vector} if 
\begin{enumerate}
\item $c_{1}\in\{0,1\}$;
\item $-m_{i}\le c_{i}\le1+m_{i}$ for all $i\in\{2,\ldots,n\}$,
\end{enumerate}
where $m_{i}:=\max\{c_{1},\ldots,c_{i-1}\}$.
\end{defn}

In other words, the first non-zero component of a coloring vector
is $1$, and for any $k\ge1$ a component equal to $-k$ or $k+1$
must be preceded by a component equal to $k$.

\begin{rem}
\label{rem:redundant}A simple inductive argument shows that if $c$
is a coloring vector, then
\begin{enumerate}
\item[(1')] $-i+1\le c_{i}\le i$ for all $i$.
\end{enumerate}
On the other hand, this condition guarantees that $c_1\in \{0,1\}$, so it is stronger than Definition~\ref{def:colVec}(1). Hence we can use (1') instead of (1) in Definition~\ref{def:colVec}. The advantage is that (1') provides finite domains for our decision variables in the constraint satisfaction problem shown in Figure~\ref{fig:CSP}.
\end{rem}

\begin{example}
The tuples $(1,0,-1,2,1,2)$ and $(0,0,1,2,3)$ are coloring vectors,
while the tuples $(1,-2,1)$, $(1,0,3)$, and $(0,2,1)$ are not coloring vectors.
\end{example}

The following is easy to verify.
\begin{prop}
\label{def:colVecAlt}The conditions of Definition~\ref{def:colVec}
are equivalent to the following conditions:
\begin{enumerate}
% \item $c_{1}\in\{0,1\}$;
\item $-i+1\le c_{i}\le i$ for all $i$;
\item for all $i\in\{2,\ldots,n\}$ there is a $j\in\{1,\ldots,i-1\}$ such that $c_{i}\le 1+c_j$;
% \item $(\exists j\in\{1,\ldots,i-1\}) (c_{i}\le 1+c_j)$ for all $i\in\{2,\ldots,n\}$;
\item for all $i\in\{2,\ldots,n\}$, either $c_{i}\ge0$ or $c_{i}\in\{-c_{1},\ldots,-c_{i-1}\}$.
%\item $c_{i}\ge0$ or $c_{i}\in\{-c_{1},\ldots,-c_{i-1}\}$ for all $i\in\{2,\ldots,n\}$.
\end{enumerate}
\end{prop}

\begin{defn}
Given a coloring vector $c\in\mathbb{Z}^{n}$, let $[i]:=\{j\mid c_{i}=c_{j}\}$,
$\mathcal{P}_{c}:=\{[i]\mid i\in\{1,\ldots,n\}\}$, and $[i]^{*}:=\{j\mid c_{j}=-c_{i}\}$.
\end{defn}

\begin{prop}
If $c$ is a coloring vector in $\mathbb{Z}^{n}$, then $\mathcal{P}_{c}$
endowed with the partial involution $[i]\mapsto[i]^{*}:\tilde{\mathcal{P}}\to\tilde{\mathcal{P}}$
for $\tilde{\mathcal{P}}=\{[i]\mid[i]^{*}\ne\emptyset\}$ is a tagged
partition of $C=\{1,\ldots,n\}$.
\end{prop}

\begin{proof}
It is easy to see that we have a partial involution on a partition
of $C$. If $c_{i}\ne0$ then $[i]$ is not a fixed point of this
partial involution. So the only possible fixed point is $\{i\mid c_{i}=0\}$
when this set is non-empty.
\end{proof}

\begin{prop} \label{prop:c_to_Pc}
The map $c\mapsto\mathcal{P}_{c}$ is a bijection
between the set of coloring vectors of $\mathbb{Z}^{n}$ and the set
of tagged partitions of $C=\{1,\ldots,n\}$.
\end{prop}

\begin{proof}
The map is clearly injective. To show that the map is also surjective,
consider a tagged partition $\mathcal{P}$ of $C$. Let $A_{0}$ be
the fixed point of the partial involution if it exists. Let $\mathcal{Q}:=\{A\in\tilde{\mathcal{P}}\mid\min(A^{*})\leq\min(A)\}$.
Note that $A_{0}\in\mathcal{Q}$ if $A_{0}$ exists. Let $(A_{1},\ldots,A_{d})$
be the ordering of the classes in $\mathcal{P\setminus\mathcal{Q}}$
according to the natural order of their minimum elements. That is,
$\min(A_{r})<\min(A_{s})$ for all $0<r<s$. Let $A_{-k}:=A_{k}^{*}$
if $A_{k}^{*}\in\mathcal{Q}$. Now every element in $\mathcal{P}$
has a unique label. Define $c\in\mathbb{Z}^{n}$ by $c_{i}=k$ if
$[i]=A_{k}$. The construction guarantees that $c$ is a coloring
vector with $\mathcal{P}_{c}=\mathcal{P}$.
\end{proof}

This proposition also implies that there is a bijection between coloring vectors in $\mathbb{Z}^n$ and polydiagonal subspaces of $\mathbb{R}^n$, given by
$$
\Delta_{{\mathcal P}_c} = \{x \in \mathbb{R}^n \mid x_i = x_j \text{ if } c_i = c_j \text{ and } x_i = -x_j \text{ if } c_i = - c_j\}.
$$

\begin{example}
\label{ex:cPbijection}
Let the partition $\mathcal{P}=\{\{1,4\},\{2\},\{3\},\{5\}\}$ be
tagged with the partial involution on $\mathcal{\tilde{P}}=\{\{1,4\},\{2\},\{5\}\}$
defined by $\{1,4\}^{*}=\{5\}$, $\{5\}^{*}=\{1,4\}$, and $\{2\}^{*}=\{2\}$.
Thus $A_{0}=\{2\}$ and $\mathcal{Q}=\{\{2\},\{5\}\}$, so $A_{1}=\{1,4\}$
and $A_{2}=\{3\}$ are the ordered elements of $\mathcal{P}\setminus\mathcal{Q}$.
Then $A_{-1}=A_{1}^{*}=\{5\}$, and the corresponding coloring vector
is $c=(1,0,2,1,-1)$. Note that
$$
\Delta_{{\mathcal P}} = \{x \in \mathbb{R}^5 \mid x_1= x_4 = -x_5, \, x_2 = 0\} = \{(a, 0, b, a, -a) \mid a,b \in \mathbb{R}\}.
$$
is the corresponding polydiagonal subspace of $\mathbb{R}^5$.
\end{example}

The coloring vectors of synchrony subspaces are characterized by the following result.

\begin{prop}
The vector $c\in\mathbb{Z}^{n}$ is a coloring vector for a partition
with an empty partial involution if and only if 
\begin{enumerate}
\item $c_{1}=1$;
\item $1\le c_{i}\le1+m_{i}$ for all $i\in\{2,\ldots,n\}$,
\end{enumerate}
where $m_{i}:=\max\{c_{1},\ldots,c_{i-1}\}$.
\end{prop}

\begin{example}
\label{ex:cSynchrony}
Let the partition $\mathcal{P}=\{\{1,4\},\{2,3\},\{5\}\}$ have the
empty involution, so $\tilde{P}=\emptyset$ and $\mathcal{Q}=\emptyset$
in the notation of Proposition~\ref{prop:c_to_Pc}. The classes of
the partition are listed with increasing minimal elements, so the
corresponding coloring vector is $c=(1,2,2,1,3)$.
Note that
$$\Delta_{{\mathcal P}}= \{x \in \mathbb{R}^5 \mid x_1 = x_4, x_2 = x_3 \} = \{(a, b,b,a,c) \mid a,b,c \in \mathbb{R}\}$$
is the corresponding synchrony subspace of $\mathbb{R}^5$.
\end{example}

It is easy to distinguish synchrony and anti-synchrony subspaces by the corresponding coloring vectors. 
If all the components of the coloring vector $c$ are positive, as in Example~\ref{ex:cSynchrony}, then $\Delta_{{\mathcal P}_c}$ is a synchrony subspace. If one or more component of $c$ is less than $1$, as in Example~\ref{ex:cPbijection}, then $\Delta_{{\mathcal P}_c}$ is an anti-synchrony subspace.

The following result is an easy consequence of the definitions.
\begin{prop}
\label{prop:inv}An element $v\in\mathbb{R}^{n}$ is in $\Delta_{\mathcal{P}_{c}}$
if and only if the conditions
\begin{enumerate}
\item $c_{i}=c_{j}$ implies $v_{i}=v_{j}$;
\item $c_{i}=-c_{j}$ implies $v_{i}=-v_{j}$;
\item $c_{i}=0$ implies $v_{i}=0$
\end{enumerate}
hold for all distinct $i$ and $j$.
\end{prop}

Let $c$ be a coloring vector in $\mathbb{Z}^{n}$. The \emph{set
of colors} of $c$ is $K_{c}:=\{|c_{i}|:1\le i\le n\}$. The \emph{signed
Kronecker delta function} is defined by
\[
\delta_{i}(j):=\begin{cases}
1, & i=j\\
-1, & i=-j\\
0, & \text{otherwise}.
\end{cases}
\]

\begin{prop}
\label{prop:basis}
Let $c$ be a coloring vector in $\mathbb{Z}^{n}$. Define $b^{(k)}:=(\delta_{k}(c_{1}),\ldots,\delta_{k}(c_{n}))$.
Then $B_{c}=\{b^{(k)}\mid k\in K_{c}\}$ is a basis for the polydiagonal subspace $\Delta_{\mathcal{P}_{c}}$.
\end{prop}

Note that $d=|K_{c}|$ is the dimension of $\Delta_{\mathcal{P}_{c}}$ and $b^{(k)}$ is the zero vector for $k\in \mathbb{N}\setminus K_c$.
The matrix $D_{c}:=\left[\begin{matrix}b^{(1)} & \cdots & b^{(d)}\end{matrix}\right]\in\{-1,0,1\}^{n\times d}$
built from the basis vectors is what \cite[Def. 3.6]{nss8} calls
a \emph{polydiagonal basis matrix}. Every row of $D_{c}$ has at most
one non-zero entry and $D_{c}^{T}$ is in reduced row-echelon form.
Given a polydiagonal basis matrix $D_{c}$, the corresponding coloring
vector can be recovered as the matrix product $c=D_{c}\iota$ with
$\iota=(1,2,\ldots,d)$. 

\begin{example}
Consider the coloring vector $c=(1,0,1,2,-2)\in\mathbb{Z}^{5}$. Then
$K_{c}=\{1,2\}$ and $B_{c}=\{(1,0,1,0,0),(0,0,0,1,-1)\}$. So the
transpose of the corresponding polydiagonal basis matrix is 
\[
D_{c}^{T}=\left[\begin{matrix}1 & 0 & 1 & 0 & 0\\
0 & 0 & 0 & 1 & -1
\end{matrix}\right]\in\{-1,0,1\}^{2\times5}.
\]
Note that $c=D_{c}\left[\begin{smallmatrix}1\\
2
\end{smallmatrix}\right]$. The corresponding tagged partition is $\mathcal{P}_{c}=\{\{1,3\},\{2\},\{4\},\{5\}\}$
with $\{2\}^{*}=\{2\}$ and $\{4\}^{*}=\{5\}$. The corresponding
polydiagonal subspace is $\Delta_{\mathcal{P}_{c}}=\span(B_{c})=\{(a,0,a,b,-b)\mid a,b\in\mathbb{R}\}\in\mathbb{R}^{5}$.
\end{example}

\begin{rem}
Recall \cite{NiSS} that a tagged partition $\mathcal{P}$ is evenly
tagged if the partial involution is fully defined and $|A|=|A^{*}|$
for all $A\in\mathcal{P}$. It is easy to see that $\mathcal{P}_{c}$
is evenly tagged if and only if $\sum_{i}b_{i}=0$ for all $b\in B_{c}$.
\end{rem}

\section{Invariant polydiagonal subspaces}

A subspace $V$ of $\mathbb{R}^{n}$ is \emph{invariant} under the matrix
$M\in\mathbb{R}^{n\times n}$ if $MV\subseteq V$. The following is
an easy consequence of Proposition~\ref{prop:inv}.

\begin{example}
The partitions corresponding to synchrony subspaces that are invariant
under the adjacency matrix of a graph are exactly the equitable partitions.
\end{example}

\begin{prop}
\label{prop:invMat}The polydiagonal subspace $\Delta_{\mathcal{P}_{c}}$\textup{
of $\mathbb{R}^{n}$} is invariant under $M\in\mathbb{R}^{n\times n}$
if and only if for all $k\in K_{c}$ the conditions
\begin{enumerate}
\item $c_{i}=c_{j}$ implies $w_{i}^{(k)}=w_{j}^{(k)}$;
\item $c_{i}=-c_{j}$ implies $w_{i}^{(k)}=-w_{j}^{(k)}$;
\item $c_{i}=0$ implies $w_{i}^{(k)}=0$
\end{enumerate}
\noindent hold for all distinct $i$ and $j$, where $w^{(k)}:=Mb^{(k)}$. 
\end{prop}

\begin{rem}
\label{rem:invMat}
In Proposition~\ref{prop:invMat}, the $k\in K_c$ condition can be replaced with $k\in\{1,\ldots, n-1\}$. This formulation is more suitable for the constraint satisfaction problem in Table~\ref{fig:CSP}. 
Since $b^{(k)}=0$ for $k\in\mathbb{N}\setminus K_c$, we can allow $k\in \{1,\ldots,n\}$ since $\{b^{(1)},\ldots,b^{(n)}\}$ is a spanning set. 
Note that $b^{(n)}$ is nonzero only if 
$c = (1, \ldots, n)$.
In this case
$\Delta_{\mathcal{P}_{c}}=\mathbb{R}^n$, so checking the conditions for $k=n$ is unnecessary as they are automatically satisfied.
% It is not needed because either $b^{(n)} = 0$ or $c = (1, \ldots, n)$. If $b^{(n)} = 0$, then $w^{(n)} = M b^{(n)} =0 \in \Delta_{\mathcal {P}_c}$ is automatic. If $c = (1, \ldots, n)$, then $\Delta_{\mathcal {P}_c} = \mathbb R^n$ is automatically $M$-invariant.
% By Proposition~\ref{prop:basis}, the basis of $\Delta_{\mathcal {P}_c}$ is $\{b^{(k)} \mid k\in K_c\}$ and $b^{(k)}=0$ for $k\not\in K_c$. 
\end{rem}

Note that the coloring vector of a synchrony subspace only has positive
components. So Conditions~(2) and (3) of Proposition~\ref{prop:invMat}
are automatically satisfied for synchrony subspaces.

\begin{rem}
The set of polydiagonal subspaces invariant under a matrix is a lattice
ordered by reverse inclusion \cite{Aguiar&Dias,AguiarDiasWeighted,AD2021}.
This ordering induces an ordering of the corresponding coloring vectors.
In this order $c^{(1)}\le c^{(2)}$ if and only if 
\begin{enumerate}
\item $c_{i}^{(1)}=c_{j}^{(1)}$ implies $c_{i}^{(2)}=c_{j}^{(2)}$;
\item $c_{i}^{(1)}=-c_{j}^{(1)}$ implies $c_{i}^{(2)}=-c_{j}^{(2)}$;
\item $c_{i}^{(1)}=0$ implies $c_{i}^{(2)}=0$
\end{enumerate}
hold for all distinct $i$ and $j$. Conditions (1)--(3) can be applied
even if $c^{(1)},c^{(2)}\in\mathbb{R}^{n}$. Using this extended order,
the conditions of Proposition~\ref{prop:invMat} can be simply stated
as $c\le w^{(k)}$ for all $k$.
\end{rem}

We often consider $M$ to be the adjacency or the Laplacian matrix of
a graph. 
However, any square matrix is the adjacency matrix of a digraph
with weighted edges, as seen in the next example \cite{NiSS}. 
\begin{example}
The Hasse diagram of the lattice of coloring vectors corresponding
to polydiagonal subspaces invariant under $M$ are shown in Figure~\ref{fig:example}.
The matrix $M$ is the in-adjacency matrix of the weighted digraph
shown. Note that $b^{(1)}=(1,0,1)$ and $b^{(2)}=(0,1,0)$ for $c=(1,2,1)$.
Hence $w^{(1)}=Mb^{(1)}=(2,0,2)\ge c$ and $w^{(2)}=Mb^{(2)}=(-1,-1,-1)\ge c$,
verifying that $\Delta_{\mathcal{P}_{c}}=\{(a,b,a)\mid a,b\in\mathbb{R}\}$
is invariant under $M$.

\begin{figure}[h]

\begin{tabular}{ccc}
\begin{tabular}{c}
\includegraphics[scale=1.2]{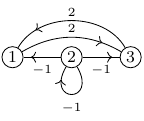}\\
\end{tabular} & %
\begin{tabular}{c}
$M=\left[\begin{matrix}0 & -1 & 2\\
0 & -1 & 0\\
2 & -1 & 0
\end{matrix}\right]$ \\
\end{tabular} & %
\begin{tabular}{c}
\begin{tikzpicture}[xscale=1.1,yscale=.8,inner sep=1pt]
\node (0) at (0,3) {$\scriptstyle(0,0,0)$};
\node (1) at (-1,2) {$\scriptstyle(1,0,-1)$};
\node (2) at (1,2) {$\scriptstyle(1,0,1)$};
\node[inner sep=1pt] (3) at (-1,1) {$\scriptstyle(1,0,2)$};
\node[inner sep=1pt,fill=green!25] (4) at (1,1) {$\scriptstyle(1,2,1)$}; 
\node[inner sep=1pt,fill=green!25] (5) at (0,0) {$\scriptstyle(1,2,3)$}; 
\draw (0)--(1);
\draw (0)--(2);
\draw (1)--(3);
\draw (2)--(3);
\draw (2)--(4);
\draw (3)--(5);
\draw (4)--(5);
\end{tikzpicture} \\
\end{tabular} \\
\end{tabular}

\caption{\label{fig:example}Weighted digraph with adjacency matrix and lattice
of coloring vectors of $M$-invariant polydiagonal subspaces. The
shading indicates the synchrony subspaces.}
\end{figure}
\end{example}

\section{Finding invariant polydiagonal subspaces using constraint programming}

\begin{figure}
\begin{framed}
\begin{flushleft}
$\bullet$ Decision variables: 
\vspace{-\baselineskip}
\[
c=(c_{1},\ldots,c_{n})\in\mathbb{Z}^{n}
\]

$\bullet$ Domains: 
\vspace{-4mm}
\[
-i+1\le c_{i}\le i \quad\text{for all \ensuremath{i}}
\tag{Rem. \ref{rem:redundant}}
\]
$\bullet$ Dependent variables:

maximums:
\vspace{-3.4mm}
\[
\begin{gathered}(m_2,\ldots,m_n)\in\mathbb{Z}^{n-1}\\
m_i:= \max\{c_{1},\ldots,c_{i-1}\} \quad\text{for all \ensuremath{i}} \end{gathered}
\]

spanning set: 
\vspace{-5mm}
\[
b^{(k)}=(b^{(k)}_{1},\ldots,b^{(k)}_{n})\in\mathbb{Z}^{n},
\quad\text{for $ k\in\{1,\ldots,n-1\}$}
\hspace{-1.2cm}
\tag{Rem. \ref{rem:invMat}}
\]
\[
b^{(k)}_{i}:=\begin{cases}
1, & c_{i}=k\\
-1, & c_{i}=-k\\
0, & \text{otherwise}
\end{cases}\qquad\text{for all \ensuremath{k} and \ensuremath{i}}
\hspace{-1cm}
\tag{Prop. \ref{prop:basis}}
\]

image of spanning set:
\vspace{-6.2mm}
\[
\hspace{9mm}
\begin{gathered}w^{(k)}=(w^{(k)}_{1},\ldots,w^{(k)}_{n})\in\mathbb{Z}^{n}\quad\text{for all $k$}\\
w^{(k)}:=Mb^{(k)}\quad\text{for all \ensuremath{k}}
\end{gathered}
\]
$\bullet$ Constraints:

coloring vector:
\vspace{-4.5mm}
\[
-m_i\le c_{i}\le1+m_i\quad\text{for \ensuremath{i\in\{2,\ldots,n\}}}
\hspace{-1cm}\tag{Def. \ref{def:colVec}(2)}
\]

$M$-invariance: 
\vspace{-4mm}
\[
\begin{aligned}\begin{gathered}c_{i}=0\end{gathered}
 & \implies w^{(k)}_{i}=0 &  & \text{for all \ensuremath{k} and \ensuremath{i}}\\
\begin{gathered}c_{i}=c_{j}\end{gathered}
 & \implies w^{(k)}_{i}=w^{(k)}_{j} &  &  \text{for all \ensuremath{k} and \ensuremath{j<i}}\\
\begin{gathered}c_{i}=-c_{j}\end{gathered}
 & \implies w^{(k)}_{i}=-w^{(k)}_{j} &  &  \text{for all \ensuremath{k} and \ensuremath{j<i}}
\end{aligned}
\hspace{-1.5cm}\tag{Prop. \ref{prop:invMat}}
\]
\end{flushleft}
\end{framed}

\caption{\label{fig:CSP} Constraint satisfaction problem for the coloring vectors $c$ of $M$-invariant polydiagonal subspaces for $M\in\mathbb{Z}^{n\times n}$.}
\end{figure}

The conditions of Definition~\ref{def:colVec}, Proposition~\ref{prop:inv},  Remark~\ref{rem:redundant}, and Proposition~\ref{prop:invMat} are well suited for constraint programming \cite{constraintProg}. Figure~\ref{fig:CSP} shows the constraint satisfaction problem for finding all invariant polydiagonal subspaces of a given matrix. 
Note that the invariant polydiagonal subspaces of $M$ and $aM$ are the same for any nonzero $a \in \mathbb{R}$. Thus, a matrix with rational entries can be replaced by an integer matrix in the computation of invariant subspaces. This requires multiplication by a common denominator of the entries. Henceforth our implementation assumes that $M$ has integer entries.

The decision variables are the components $c_{1},\ldots,c_{n}$ of the coloring vector $c$. The domains of the decision variables are determined by Remark~\ref{rem:redundant}. We also use the dependent (inferred) variables $m_2,\ldots,m_n$, $b^{(1)},\ldots,b^{(n-1)}$, and $w^{(1)},\ldots,w^{(n-1)}$. 
The subtle reasons for the range $k \in \{1, \ldots, n-1\}$ for $b^{(k)}$ rely on Remark~\ref{rem:invMat}. Since the $b^{(k)}$ are determined by decision variables during run time, even the zero vectors need to be defined. Also, in order to make the algorithm more efficient, $b^{(n)}$ is not needed, hence not defined.  

% and again the $k = n$ constraint is not needed.

%{
%Note that $m_1$ is not defined, and $m_2 = c_1$. Furthermore, $\max(c_1, \ldots, c_n)$ is the dimension of the space $\Delta_{\mathcal {P}_c}$, and $b^{(k)} = 0$ if $k >\dim(\Delta_{\mathcal {P}_c})$.  
%
%The conditions of Proposition~\ref{prop:invMat} depend on $K_{c}$.
%This set $K_{c}$ of colors must be a subset of $\{1,\ldots,n\}$. In fact the condition of Proposition~\ref{prop:invMat} can be modified by requiring $k\in\{1,\ldots,n-1\}$ instead of $k\in K_{c}$. 
%Note that $b^{(k)}$ is the zero vector for all $k\in\{1,\ldots,n\}\setminus K_{c}$.
%So the conditions of Proposition~\ref{prop:invMat} are guaranteed for these additional $k$ values.
%We do not define $b^{(n)}$ and $w^{(n)}$ because they are nonzero only when $c = (1,\ldots, n)$. In this special case, $\Delta_{\mathcal{P}_c} = \mathbb R^n$ is automatically $M$-invariant.
%
%This modification is used in our code.
%}

We developed code that solves this constraint satisfaction problem using four Python interfaces to three different solvers. Two of them are available in the appendices. All four programs are available at our companion website \cite{NSS10_companion}.

In these programs, the dependent variables can be actual variables or expressions, depending on the constraint solver. After a solver finds all possible coloring vectors, the corresponding invariant polydiagonal spaces can be built using Proposition~\ref{prop:c_to_Pc}. 

The best performing version for smaller problems uses DOcplex, the Python interface to IBM's CP Optimizer \cite{cp_optimizer_2023}. The versions using the CP-SAT solver \cite{cpsatlp} handle larger or more complex problems better. We use the CP-SAT solver through two interfaces. One of them uses the native ORtools interface, the other uses the CPMpy \cite{guns2019increasing} interface. The ORtools version performs better if we use the conditions of Proposition~\ref{def:colVecAlt}. The version using the Z3 Theorem Prover \cite{z3} is significantly slower. This is probably because Z3 is designed to solve more generic problems. We include the Z3 version because it is easily accessible, user friendly, and has a non-commercial license.

As an initial test, we ran our code without the constraints that force $M$-invariance, which we can think of as letting $M$ be the zero matrix. This finds coloring vectors in $\mathbb{Z}^n$. The number of such coloring vectors is the Dowling number $p_n$ \cite[A007405]{oeis} as was shown in \cite{NiSS}. Our code successfully verifies these counts. Table~\ref{tab:performcol} shows the performance of our implementations. These values show that the performance of the Z3 code is inferior even for this simpler problem compared to the other two implementations.

\begin{table}[ht]
\begin{tabular}{|l|c|c|c|c|c|c|c|c|c|c|}
\hline 
$n$ & 1 & 2 & 3 & 4 & 5 & 6 & 7 & 8 & 9 & 10  \\
\hline 
$p_n$ & 2 & 6 & 24 & 116 & 648 & 4088 & 28640 & 219920 & 1832224 & 16430176  \\
\hline 
DOcplex & 0.4 & 0.4 & 0.4 & 0.4 & 0.5 & 1.5 & 8.1 & 61 & 501 & 4531   \\
\hline 
ORtools & 0.4 & 0.3 & 0.3 & 0.3 & 0.4 & 1.0 & 6.1 & 57 & 586 & 6427   \\
\hline 
Z3 & 0.3 & 0.3 & 0.3 & 0.4 & 1.1 & 12 & 532 & 47513 &  &    \\
\hline
\end{tabular}

\caption{\label{tab:performcol}
The number $p_n$ of coloring vectors in $\mathbb{Z}^n$, and the time in seconds to compute them with three different solvers. Since the ORtools interface to CP-SAT outperforms the CPMpy interface, the CPMpy times are omitted from the table.
}
\end{table}

\section{Comparison to existing approaches}

No algorithm, other than brute force, is currently available to compute all invariant polydiagonal subspaces. Our constraint programming approach is the first practical way to handle medium-size matrices.  
 
The first useful algorithm to compute only the invariant synchrony subspaces was developed in \cite{KameiLattice} and further improved in \cite{NSS7} as the split and cir algorithm. These algorithms rely on the computation of the coarsest invariant refinement (cir) of a partition. The synchrony subspace corresponding to the cir is the invariant synchrony subspace with the smallest dimension that contains the original synchrony subspace.

The split and cir algorithm starts with the singleton partition $\{\{1,\ldots,n\}\}$, corresponding to the coloring vector $c = (1,\ldots,1)$. 
The one-dimensional subspace $\Delta_{\mathcal{P}_c}$ is not necessarily invariant. The queue of candidates is initialized with this partition. For each candidate in the queue, the cir is computed and added to the list of partitions producing invariant synchrony subspaces.  For each new invariant partition, each class is split into two classes in all possible ways, which are added to the candidate queue. The queue will eventually be empty, terminating the algorithm.

Our constraint programming code greatly outperforms the split and cir algorithm. A theoretical algorithm for computing invariant polydiagonal subspaces using eigenvectors of the matrix $M$ is described in \cite{Aguiar&Dias,AD2021}.
Since no implementation is available, it is hard to compare the performance of this approach to ours. One advantage of our approach is that it uses exact arithmetic. 

Table~\ref{tab:perform} shows how the performance 
of the split and cir algorithm
\cite{NSS7} compares to the constraint approach for the Laplacian
matrix of the cycle graph $C_{n}$. 
Our computations were done on a AMD EPYC 7542 processor using a single  
2.9 GHz CPU and 16G RAM. 
In many instances, the use of constraint programming reduces the time of computation from weeks to seconds.

\begin{table}[ht]
\begin{tabular}{|l|c|c|c|c|c|c|c|c|c|c|}
\hline 
$n$ & 5 & 10 & 15 & 20 & 25 & 30 & 35 & 40 & 45 & 50 \\
\hline 
\multicolumn{1}{l}{number of subspaces} & \multicolumn{1}{c}{} & \multicolumn{1}{c}{} & \multicolumn{1}{c}{} & \multicolumn{1}{c}{} & \multicolumn{1}{c}{} & \multicolumn{1}{c}{} & \multicolumn{1}{c}{} & \multicolumn{1}{c}{} & \multicolumn{1}{c}{} & \multicolumn{1}{c}{} \\
\hline 
synchrony & 7 & 19 & 27 & 45 & 33 & 77 & 51 & 95 & 83 & 96 \\
\hline 
polydiagonal & 13 & 47 & 51 & 121 & 64 & 197 & 99 & 269 & 161 & 250 \\
\hline 
\multicolumn{1}{l}{times (sec)} & \multicolumn{1}{c}{} & \multicolumn{1}{c}{} & \multicolumn{1}{c}{} & \multicolumn{1}{c}{} & \multicolumn{1}{c}{} & \multicolumn{1}{c}{} & \multicolumn{1}{c}{} & \multicolumn{1}{c}{} & \multicolumn{1}{c}{} & \multicolumn{1}{c}{} \\
\hline 
split \& cir: sync. & 0 & 0 & 0 & 9 & 401 & 22631 &  &  &  & \\
\hline 
DOcplex: sync. & 0.4 & 0.4 & 0.5 & 0.7 & 1.2 & 2.1 & 3.3 & 5.4 & 8.0 & 12\\
DOcplex: polydiag. & 0.4 & 0.4 & 0.6 & 1.2 & 2.4 & 5.5 & 11 & 26 & 57 & 124\\
\hline 
ORtools: sync. & 0.4 & 1.5 & 6.1 & 11 & 18 & 34 & 54 & 107 & 195 & 333\\
ORtools: polydiag. & 0.5 & 2.3 & 7.9 & 19 & 38 & 91 & 164 & 348 & 613 & 1051\\
\hline 
\end{tabular}

\caption{\label{tab:perform}
Times to compute the synchrony and polydiagonal subspaces invariant under the Laplacian of $C_{n}$. Only synchrony subspaces are computed with the split and cir algorithm.}
\end{table}

Polydiagonal subspaces that are invariant under a collection of matrices
are also often useful as was shown in \cite{NSS7}. The constraint
programming approach for finding invariant polydiagonal subspaces
easily generalizes to this setting. It only requires
the constraints corresponding to the conditions of Proposition~\ref{prop:invMat}
for each matrix in the collection.

\section{Examples}

Let $\Aut(G)$ be the automorphism group and $M$ be the adjacency
matrix of the simple graph $G$. Define $\Gamma:=\Aut(G)\times\mathbb{Z}^{*}$,
where and $\mathbb{Z}^{*}=\{-1,1\}$ is the unit group of $\mathbb{Z}$.
For $(\phi,\sigma)\in\Gamma$ and $x\in\mathbb{R}^{n}$, define $(\phi,\sigma)x:=(\sigma x_{\phi(1)},\ldots,\sigma x_{\phi(n)})$.
Then $\Gamma$ acts on the set of $M$-invariant polydiagonal subspaces
by $(\phi,\sigma)V:=\{(\phi,\sigma)x\mid x\in V\}$. The point stabilizer
of such a subspace $V$ is 
\[
\pStab(V):=\{(\phi,\sigma)\in\Gamma\mid(\forall x\in V)\,(\phi,\gamma)x=x\}.
\]
The fixed point subspace $\Fix(\Theta)$ of a subgroup $\Theta$ of
$\Gamma$ is the maximal $M$-invariant subspace whose point stabilizer
is $\Theta$. We call $\Fix(\pStab(V))$ the fixed point subspace
corresponding to $V$. If $\Fix(\pStab(V))\ne V$, then $V$ is an
\emph{Anomalous Invariant Subspace} (AIS). In other words, an AIS
is an $M$-invariant polydiagonal subspace that is not a fixed point subspace of
the $\Gamma$ action. Although fixed point subspaces can be found
by analyzing representations of the $\Gamma$ action \cite{NSS3},
finding each AIS is very difficult without constraint programming. 
\begin{example}
Our DOcplex and ORtools codes both find the 240 invariant polydiagonal subspaces
for the adjacency matrix of the Petersen graph in a few seconds. 
These subspaces are in 22
orbits under the $\Gamma$ action. The number of invariant anti-synchrony
subspaces is 147. Two orbits contain AIS.
\end{example}

\begin{example}
The invariant polydiagonal subspaces for the graph Laplacian $L$ for the path graph with 100 vertices can be computed by our DOcplex and ORtools codes in 7.4 and 31.9 minutes, respectively. There are 20 $L$-invariant subspaces, each in a singleton group orbit of the $\Gamma \cong \mathbb{Z}_2 \times \mathbb{Z}_2$ action. Only 4 of these invariant subspaces are fixed point subspaces, with coloring vectors
\[
% c_0 = 
(0, \ldots, 0), \quad
%c_8 = 
(1, \ldots, 50, -50, \ldots, -1), \quad 
% c_{18} = 
(1, \ldots,50, 50, \ldots,  1), \quad
% c_{19} = 
(1, \ldots, 100) .
\]
Two of the invariant subspaces, for example, that are AIS have the coloring vectors
\[
(1, -1, -1, 1, \ldots , 1,-1, -1, 1), \quad
(1, 2, 2, 1, \ldots , 1, 2, 2, 1),
\]
where the groups of four are repeated.

These results could be reproduced using the eigenvectors of $L$, using the algorithm described in \cite{Aguiar&Dias,AD2021}.  
This can be done because the eigenvectors $v_k\in \mathbb{R}^{100}$ for $k \in \{0, \ldots, 99\}$ are known in closed form with components $(v_k)_i = \cos((i-\frac 1 2) k \pi/100)$. 
\end{example}

% \url{https://math.stackexchange.com/questions/4128143/eigenvectors-of-laplacian-matrix-of-path-graph} 

\begin{figure}
\includegraphics[scale=0.35]{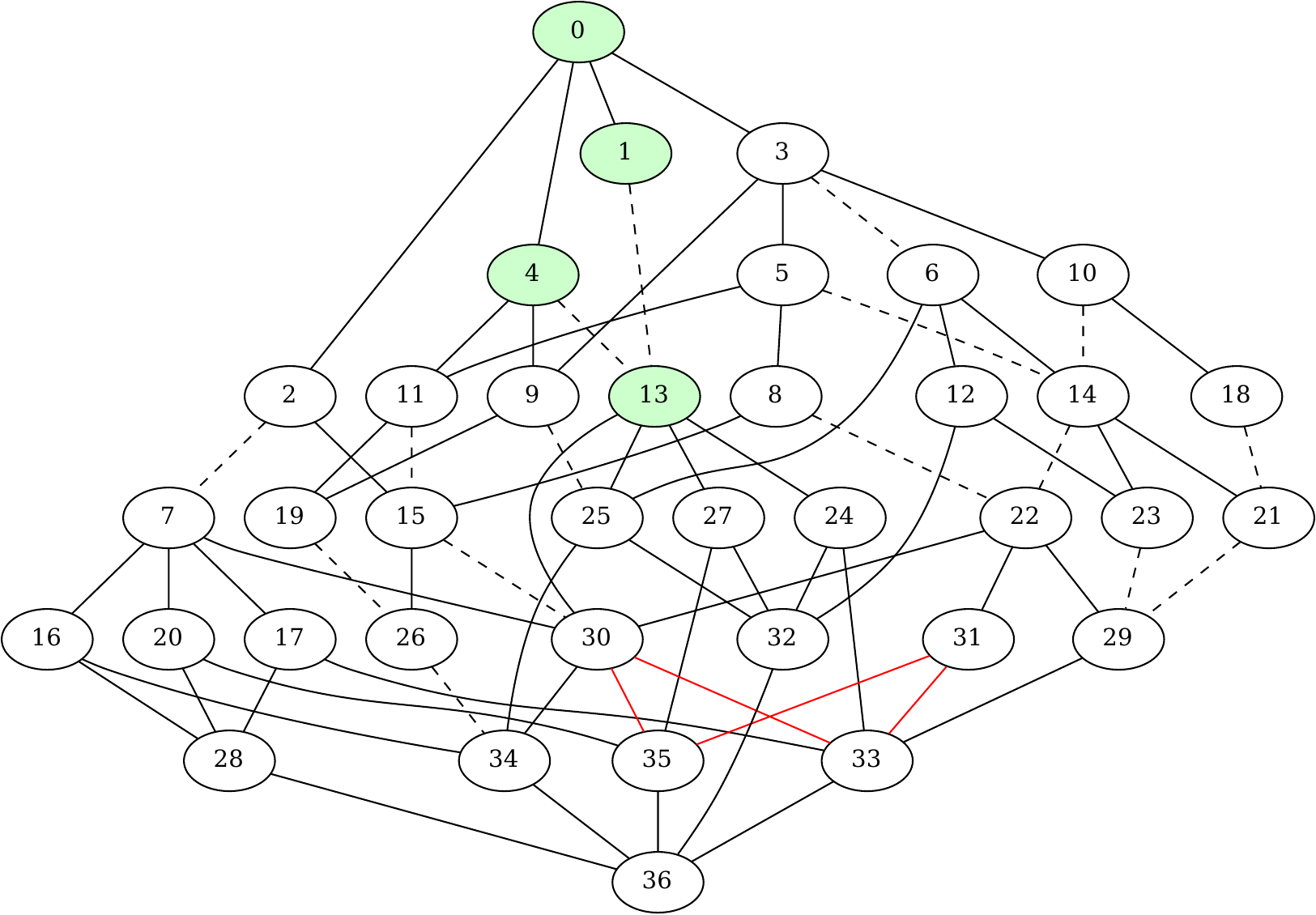}

\caption{\label{fig:BuckyHasse}The Hasse diagram of the poset of $\Gamma$-orbits of $M$-invariant synchrony subspaces for the Buckyball graph. The shaded orbits are described in Figures~\ref{fig:w1012} 
through \ref{fig:1012subs}. The relationship between orbits 30, 31, 33, and 35, as indicated by the red figure eight, shows that this poset is not a lattice.}
\end{figure}

\begin{figure}
\setlength{\tabcolsep}{2.5mm}
%scale was .9 and .125.  Jim changed to 1 and 1.35 on 2024-05-21
\begin{tabular}{cc}
\begin{tabular}{c}
\includegraphics[scale=1]{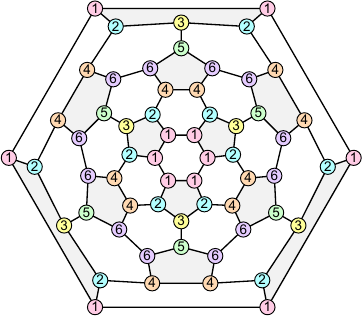}\\
\end{tabular} & %
\begin{tabular}{c}
\includegraphics[scale=0.135]{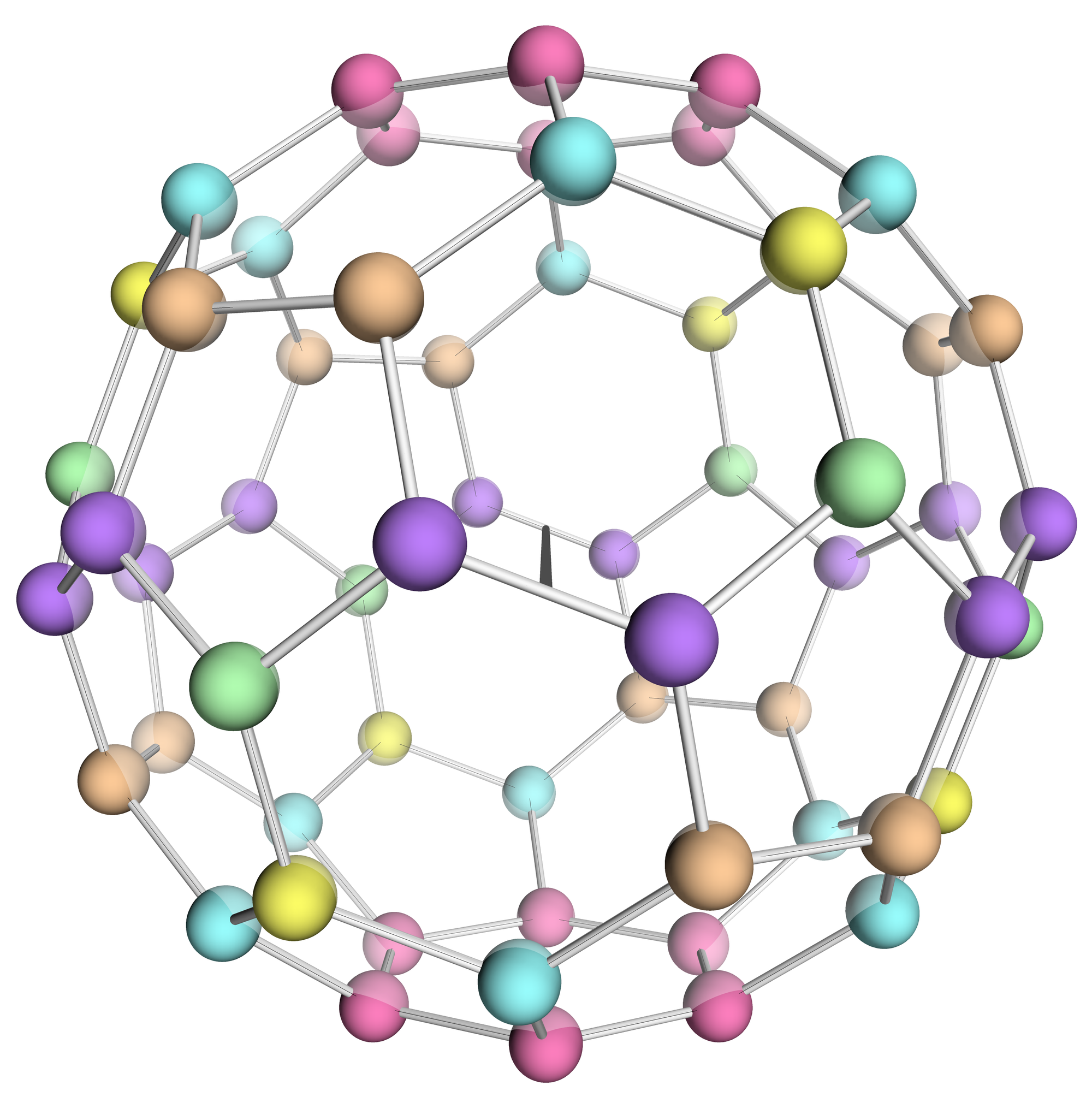}\\
\end{tabular}\\
\end{tabular}

\caption{\label{fig:w1012}The invariant synchrony subspace $W$ of Example~\ref{exa:Bucky}.
The 2D embedding uses an azimuthal equidistant projection of the 3D
embedding. The fixed point subspace $W$ has point stabilizer isomorphic
to $D_{3d}\protect\cong D_{3}\times\mathbb{Z}_{2}$. The $D_{3}$ symmetry
can be seen with the help of the shading in the 2D picture of $W$,
while the $\mathbb{Z}_{2}$ symmetry is the $180^{\circ}$ rotation
around the black axis shown in the 3D embedding of $W$. }
\end{figure}

\begin{figure}
\setlength{\tabcolsep}{1.5pt}
% scale of buckyQ was 1.2.  Jim changed it to 1.1
% This figure was too wide, in Jim's opinion.  He changed the Mc to a smallmatrix.  We could also take out the "M_c = " to get more space!
\begin{tabular}{ccc}
\begin{tabular}{c}
\includegraphics[scale=1]{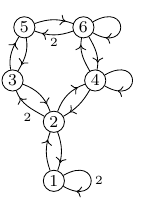}\\
\end{tabular} & %
\begin{tabular}{c}
$M_{c}=\left [ 
\begin{matrix}
2 & 1 & 0 & 0 & 0 & 0\\
1 & 0 & 1 & 1 & 0 & 0\\
0 & 2 & 0 & 0 & 1 & 0\\
0 & 1 & 0 & 1 & 0 & 1\\
0 & 0 & 1 & 0 & 0 & 2\\
0 & 0 & 0 & 1 & 1 & 1
\end{matrix} 
\right ]$\\
\end{tabular} & %
\begin{tabular}{c}
\begin{tikzpicture}[xscale=2.5,yscale=.9,inner sep=1pt]
\node (0) at (1,3) {$\scriptstyle(0,0,0,0,0,0)$};
\node (1) at (0,2) {$\scriptstyle(0,0,1,-1,1,0)$};
\node (2) at (1,2) {$\scriptstyle(1,-1,-1,-1,1,1)$};
\node[inner sep=1pt,fill=green!25] (3) at (2,2) {$\scriptstyle(1,1,1,1,1,1)$};
\node[inner sep=1pt,fill=green!25] (4) at (1,1) {$\scriptstyle(1,2,2,2,1,1)$};
\node[inner sep=1pt,fill=green!25] (5) at (2,1) {$\scriptstyle(1,2,3,1,2,1)$};
\node[inner sep=1pt,fill=green!25] (6) at (1,0) {$\scriptstyle(1,2,3,4,5,6)$};
\draw (0)--(1);
\draw (0)--(2);
\draw (0)--(3);
\draw[dashed] (1)--(6);%\draw[dashed] (1)--(6);
\draw[dashed] (2)--(4);%\draw[dashed] (2)--(4);
\draw[dashed] (3)--(4);%\draw[dashed] (3)--(4);
\draw[dashed][dashed][dashed] (3)--(5);%\draw[dashed] (3)--(5);
\draw[dashed][dashed] (4)--(6);%\draw[dashed] (4)--(6);
\draw[dashed] (5)--(6);%\draw[dashed] (5)--(6);
\end{tikzpicture}\\
\end{tabular}\\
\end{tabular}

\caption{\label{fig:QuotientDigraph}
The quotient digraph $G_c$ of $W$ from Figure~\ref{fig:w1012} with its
adjacency matrix $M_{c}$ and lattice of coloring vectors of $M_{c}$-invariant
polydiagonal subspaces. The four shaded coloring vectors of synchrony
subspaces correspond to the shaded orbits of Figure~\ref{fig:BuckyHasse}.
}
\end{figure}

%\note{Jim: Here's an explanation of the newold in Figure 6.3.  The *lattice* in Figure 4.1 has solid lines, but the *poset* in 6.1 has solid and dashed lines. I suggest that the *lattice* in 6.3 have only solid lines.  In the text we say ``Orbits connected by a dashed line (in figure 6.1) have the same conjugacy class of point stabilizers." If we have dashed lines in figure 6.3 that don't agree with the dashed lines in figure 6.1, we should explain it.  It is subtle and not very important.  That being said, I am fine with either the ``old" or new version of the caption.}

\begin{figure}
\setlength{\tabcolsep}{5mm}
\begin{tabular}{cc}
\includegraphics[scale=1]{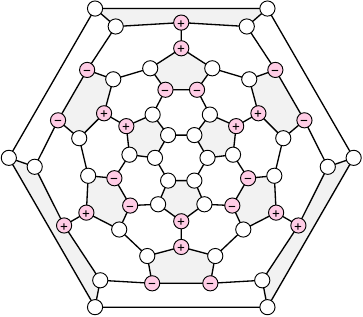} & \includegraphics[scale=1]{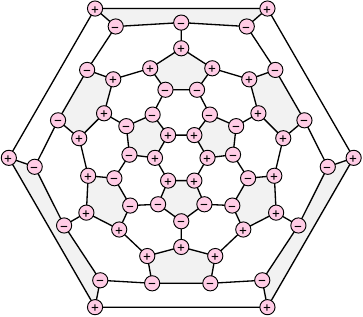}\\
$(0,0,1,-1,1,0)$ & $(1,-1,-1,-1,1,1)$\\
\vspace{-.1cm} \\ % Jim added this line for more space between rows
\includegraphics[scale=1]{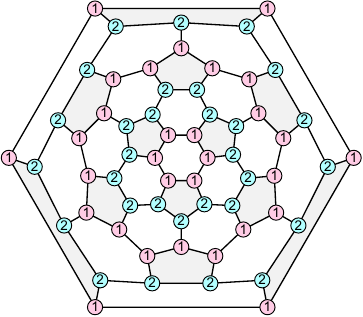} & \includegraphics[scale=1]{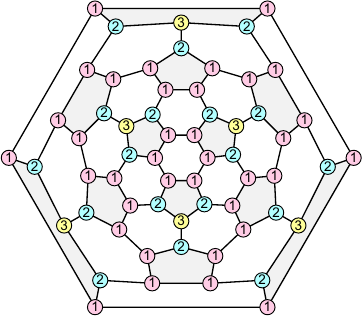}\\
$(1,2,2,2,1,1)$ & $(1,2,3,1,2,1)$\\
\end{tabular}

\caption{\label{fig:1012subs}The $M$-invariant AIS that have the same point
stabilizer as $W$. The coloring vectors below the figures indicate
the corresponding $M_{c}$-invariant polydiagonal subspaces shown
in Figure~\ref{fig:QuotientDigraph}.}
\end{figure}

\begin{example}
\label{exa:Bucky}The 60-vertex Buckyball graph $G$ models the Buckminsterfullerene
molecule Carbon 60. The standard embedding of the Buckyball graph
into 3-space is the skeleton of the truncated icosahedron, familiar
as the soccer ball. The symmetry group of this embedding is the 120-element
icosahedral group $I_{h}$ using Sch\"onflies notation \cite{Tinkham}.
Note that $I_{h}$ is a subgroup of the orthogonal group $O(3)$,
and $I_{h}\cong\Aut(G)\cong A_{5}\times\mathbb{Z}_{2}$. Thus, $\Gamma=\Aut(G)\times\mathbb{Z}^{*}$
has 240 elements. Let $M$ be the adjacency matrix of $G$. Our ORtools
code finds the 1244 $M$-invariant polydiagonal subspaces in 6.7 hours. The DOcplex implementation failed to find all solutions in 14 days.
These polydiagonal subspaces are partitioned into 95 group
orbits under the $\Gamma$ action. The same code finds the 340 $M$-invariant
synchrony subspaces, partitioned into 37 group orbits, in under 3
minutes. The partially ordered set of $M$-invariant synchrony subspaces
is shown in Figure~\ref{fig:BuckyHasse}. Orbits 0 and 36 are both
singletons; orbit 0 contains the fully synchronous subspace with coloring
vector $(1,\ldots,1)$, and orbit 36 contains $\mathbb{R}^{60}$ with
the coloring vector $(1,\ldots,60)$. 

Orbits connected by a dashed line in Figure~\ref{fig:BuckyHasse}
have the same conjugacy class of point stabilizers. The 16 group orbits
with a dashed line pointing down contain synchrony subspaces that
are AIS, and the other 21 group orbits contain synchrony subspaces
that are fixed point subspaces. 

Now we focus on the synchrony subspace $W$ shown in Figure~\ref{fig:w1012},
which is a fixed point subspace and a member of orbit 13 of Figure~\ref{fig:BuckyHasse}.
The quotient digraph $G_c$ of $W$, where $c$ is the coloring vector of $W$, is shown in Figure~\ref{fig:QuotientDigraph}. An arrow $i\!\!\shortrightarrow\!\!j$ with weight $a$ in the quotient
digraph indicates that every vertex with color $j$ is connected to
$a$ vertices with color $i$. The adjacency matrix of the quotient
digraph is $M_{c}:=D_{c}^{+}MD_{c}\in\mathbb{R}^{6\times 6}$ \cite{nss8}. Note that the pseudoinverse $D_{c}^{+}=(D_{c}^{T}D_{c})^{-1}D_{c}^{T}$ is easy to compute because $D_{c}^{T}D_{c}$ is diagonal. The quotient digraph, the matrix $M_{c}$, and the lattice of the six $M_{c}$-invariant subspaces are shown in Figure~\ref{fig:QuotientDigraph}. 
The automorphism group of the quotient digraph is trivial, and
the only fixed point subspaces of the $\Aut(G_c)\times\mathbb{Z}^*\cong \mathbb{Z}_2$ action on $\mathbb{R}^6$ are $(0, \ldots, 0)$ and $(1, \ldots, 6).$ This lattice is isomorphic to the lattice of $M$-invariant polydiagonal subspaces that are also subspaces of $W$. Figure~\ref{fig:1012subs} shows the four of these subspaces that are AIS.
The remaining three of these subspaces, the trivial subspace, the constant subspace, and $W$ itself, are fixed point subspaces.
\end{example}

\section{Conclusions}
Polydiagonal (synchrony or anti-synchrony) subspaces invariant under a matrix have many applications in network theory and dynamical systems, especially coupled cell networks. 
We have developed an approach based on coloring vectors which makes it possible to formulate constraint satisfaction problems for finding invariant polydiagonal subspaces. We have shown that solving the resulting problems with existing state-of-the-art constraint solvers greatly outperforms split and cir, the best previously known 
algorithm for computing synchrony subspaces. 
Additionally, we are now able to compute all anti-synchrony subspaces.  
The speedup of our new code is remarkable.  For example, a problem that took weeks for the split and cir code can now be done in seconds.

The constraint satisfaction problem of Figure~\ref{fig:CSP} was translated into programs that use Python interfaces to several solvers.  
Our four Python programs are available at the GitHub repository~\cite{NSS10_companion}, including the two listings featured in the appendices.

In our examples, we enforced mathematically equivalent constraints via a variety of methods.
We used different constraint solvers and choices of hyper-parameters.
We tested our implementations on problems of different sizes and complexity.
Based on our experiments, we make the following recommendations for finding invariant polydiagonal subspaces 
with our code found at GitHub.
\begin{itemize}
\item
{\bf ORtools:}
The native Python interface to CP-SAT in ORtools (Google) \cite{cpsatlp} is the best choice for complex problems,
but has the most difficult syntax.
\item
{\bf CPMpy:}
If modification of the code is desired, then use the CPMpy \cite{guns2019increasing} interface to CP-SAT.
It is easier to understand and executes CP-SAT
nearly as fast as the native ORtools interface.
This interface can also access other solvers.
\item
{\bf DOcplex:}
For relatively simple problems the DOcplex (IBM) \cite{cp_optimizer_2023} code is generally fastest. 
%It is also a user-friendly interface, 
It also has a user-friendly syntax, 
but may not handle more complex problems.
The use of DOcplex requires the installation of CP Optimizer. A license is required but it is free for academic users.
\item
{\bf Z3:}
The Z3 Theorem Prover (Microsoft) \cite{z3} code is easy to understand, but the performance does not match the other solvers.
\end{itemize}

The ability to find polydiagonal subspaces that are invariant under large matrices opens up exciting possibilities 
for studying the dynamics of complex coupled cell networks. 
For example, our buckyball results have implications for understanding the vibrations of the Carbon 60 molecule.
An extension of the current paper would be the computation of polydiagonal subspaces 
that are invariant under a set of matrices \cite{NSS7}.
This computation would be useful in many applications.

The performance of our new code suggests that other combinatorially difficult problems in network theory could be solved using constraint programming.

\section*{Statements and Declarations}
% The authors have no conflicts of interest or competing interests. 
Some of the computations were performed on Northern Arizona University's Monsoon computing cluster, funded by Arizona's Technology and Research Initiative Fund.

\bibliographystyle{plain}
\bibliography{bib}

\pagebreak

\section*{Appendix A}
\begin{figure}[h]

\begin{lstlisting}[numbers=left,numberstyle={\footnotesize},basicstyle={\footnotesize\ttfamily},frame=single]
from docplex.cp.model import CpoModel

# read integer matrix from file
Mat=[[int(x) for x in r.split()] for r in open('M.txt') if r.strip()]
n=len(Mat)

mdl=CpoModel(name='polydiagonal')
# coloring vector decision variables with domains
c=[mdl.integer_var(min=-i,max=i+1,name=f"c_{i}") for i in range(n)]
# maximums
m={i:mdl.max([c[j] for j in range(i)]) for i in range(1,n)}
# spanning set
b={k:[(c[i]==k)-(c[i]==-k) for i in range(n)] for k in range(1,n)}
# image of spanning set
w={k:[mdl.sum(Mat[i][j]*b[k][j] for j in range(n) if Mat[i][j]!=0) 
                                for i in range(n)] for k in b}
# constraints
# coloring vector
for i in m:
  mdl.add(-m[i]<=c[i],c[i]<=1+m[i])
# M-invariance
for k in b:
  for i in range(n):
    mdl.add(mdl.if_then(c[i]==0,w[k][i]==0))
    for j in range(i):
      mdl.add(mdl.if_then(c[i]== c[j],w[k][i]== w[k][j]))
      mdl.add(mdl.if_then(c[i]==-c[j],w[k][i]==-w[k][j]))

# find all solutions
siter=mdl.start_search(SearchType='DepthFirst',Workers=1,
                       LogVerbosity='Quiet')
for sol in siter:
  print(*[sol[c[i]] for i in range(n)])
\end{lstlisting}
\caption{
\label{code:CP}
DOcplex code for computing coloring vectors of $M$-invariant polydiagonal subspaces.}
\end{figure}

Figure~\ref{code:CP} lists our DOcplex \cite{cp_optimizer_2023} code for solving the constraint satisfaction problem in Figure~\ref{fig:CSP}. The code is available at \cite{NSS10_companion}. DOcplex is a Python interface to the CP Optimizer. The code reads in the integer matrix $M$ from the file \texttt{M.txt} and writes the coloring vectors of all of the $M$-invariant polydiagonal subspaces to standard output. Since the indices of a Python list start at 0, the components of $c$ are indexed by $i \in  \{0, \ldots, n-1\}$. Every dependent variable is encoded as an expression.
In the expressions for $w^{(1)},\ldots,w^{(n-1)}$ we avoid zero multiples of a variable. Since $M$ is often a sparse matrix, this simplifies our constraints. Since our goal is not only finding a single solution but finding all of them, we modify the default hyper-parameters by selecting the `DepthFirst' search type and using only one worker. We keep the default presolver, since it results in faster execution. 

\newpage

\section*{Appendix B}

\begin{figure}[h]
\begin{lstlisting}[numbers=left,numberstyle={\footnotesize},basicstyle={\footnotesize\ttfamily},frame=single]
from cpmpy import *

# read integer matrix from file
Mat=[[int(x) for x in r.split()] for r in open('M.txt') if r.strip()]
n=len(Mat)

mdl=Model()
# coloring vector decision variables with domains
c=[intvar(-i,i+1,name=f"c_{i}") for i in range(n)]
# spanning set
b={k:[intvar(-1,1,name=f"b_{k}_{i}") for i in range(n)] 
   for k in range(1,n)}
for k in b:
  for i in range(n):
    mdl+=(c[i]== k).implies(b[k][i]== 1)
    mdl+=(c[i]==-k).implies(b[k][i]==-1)
    mdl+=((c[i]!=k) & (c[i]!=-k)).implies(b[k][i]==0)
# image of spanning set
w={k:[sum(Mat[i][j]*b[k][j] for j in range(n) if Mat[i][j]!=0)
   for i in range(n)] for k in b}

# constraints
# coloring vector
for i in range(1,n):
    mdl+=any([c[i]<=1+c[j] for j in range(i)])
    mdl+=any([c[i]>=0]+[c[i]==-c[j] for j in range(i)])
# M-invariance
for i in range(n):
  mdl+=(c[i]==0).implies(all([w[k][i]==0 for k in b]))
  for j in range(i):
    mdl+=(c[i]== c[j]).implies(all([w[k][i]== w[k][j] for k in b]))
    mdl+=(c[i]==-c[j]).implies(all([w[k][i]==-w[k][j] for k in b]))

# find all solutions
mdl.solveAll(display=lambda: print(*[a.value() for a in c]))
\end{lstlisting}
\caption{
\label{code:CPMpy}
CPMpy code for computing coloring vectors of $M$-invariant polydiagonal subspaces.}
\end{figure}

Figure~\ref{code:CPMpy} lists our CPMpy \cite{guns2019increasing} code for solving the constraint satisfaction problem in Figure~\ref{fig:CSP}. The code is available at \cite{NSS10_companion}. CPMpy is a Python library with access to several solvers. The default solver is CP-SAT. Although this CPMpy code does not reach the speed of similar code using the ORtools interface, the syntax of CPMpy is significantly simpler. It is likely that the speed difference can be eliminated by solver parameter tuning. The native ORtools code is also available at \cite{NSS10_companion}.

In this code, we use Proposition~\ref{def:colVecAlt} to enforce that $c$ is a coloring vector. As a result, the dependent variables $m_2,\ldots,m_n$ are not computed. The coloring vector conditions are encoded using logical or-statements. We made this choice because experimentation showed that this encoding performs better using the CP-SAT solver.

The dependent variables $w^{(1)},\ldots,w^{(n-1)}$ are also encoded using expressions instead of actual variables. As in the DOcplex code, in these expressions we avoid zero multiples of a variable. 

In this code we handle the $k$-dependence of the constraints enforcing the $M$-invariance in a different but logically equivalent formulation to that in Figure~\ref{fig:CSP}. Embedding the $k$-dependence within the constraints results in fewer but more complex constraints compared to the DOcplex code. Experiments show that this results in a significant performance improvement. We believe that the reason for this is that we do not utilize the CP-SAT presolver. Experiments show that the presolver slows down the computation for our problem.
\\

\end{document}